\documentclass{amsart}

\usepackage{pdfsync}
\usepackage{amsmath,mathrsfs,amssymb,txfonts}

\usepackage{graphicx,xcolor}
\usepackage[all]{xy}

\usepackage{hyperref}
\hypersetup{
    colorlinks=true,       
    linkcolor=blue,          
    citecolor=blue,        
    filecolor=blue,      
    urlcolor=blue           
}

\newtheorem{thm}{Theorem}
\newtheorem{cor}[thm]{Corollary}
\newtheorem{lem}[thm]{Lemma}
\newtheorem{prop}[thm]{Proposition}
\theoremstyle{definition}

\newtheorem{rem}[thm]{Remark}
\newtheorem{exam}[thm]{Example}

\numberwithin{equation}{section}

\makeindex 

\begin{document}

\title{$q$-\'etale covers of cyclic $p$-gonal covers}
\author[A. Carocca]{\'Angel Carocca}
\email{angel.carocca@ufrontera.cl}

\author[R. A. Hidalgo]{Rub\'en A. Hidalgo}
\email{ruben.hidalgo@ufrontera.cl}

\author[R. E. Rodr\'{\i}guez]{Rub\'{\i} E. Rodr\'{\i}guez}
\email{rubi.rodriguez@ufrontera.cl}

\address{Departamento de Matem\'atica y Estad\'{\i}stica, Universidad de La Frontera. Temuco, Chile}

\subjclass{14H40, 14H30}
\keywords{Riemann Surfaces, Coverings, Jacobian, Prym variety}

\thanks{The authors were partially supported by Grants Fondecyt 1190001, 1190991 and 1200608}
\begin{abstract}
In this paper we study the Galois group of the Galois cover of the composition of a $q$-cyclic \'etale cover and a cyclic $p$-gonal cover for any odd prime $p$. Furthermore, we give properties of isogenous decompositions of certain Prym and Jacobian varieties associated to intermediate subcovers  given by subgroups.
\end{abstract}

\maketitle


\section{Introduction}

Let $\mathcal{X}$ be a compact Riemann surface  and  $  \varphi: \mathcal{X} \rightarrow \mathbb{P}^{1}$  a cover of degree  $p$. The problem of determining the structure of the  Galois group  of $  \varphi $ in general was originally considered by O. Zariski \cite{Zariski;1978}.
This problem has since been considered by many authors and some results are known, see for instance   \cite{Allcock;Hall;2010, Artebani;Pirola;2005, Guralnick;1995, Guralnick;Shareshian;2007, Magaard;Volklein;2004}.
In this direction,  a more general problem is  to consider   a sequence of covers $\mathcal{Y} \xrightarrow{\quad {\psi}} \mathcal{X} \xrightarrow {\quad \varphi} \mathbb{P}^{1}$ of compact Riemann surfaces, and  to study the Galois group of the composite cover $\varphi \circ \psi:\mathcal{Y}\rightarrow \mathbb{P}^{1}$. 
Some results on this problem, considering special properties of the covers $\varphi$ and $\psi$, can be found for instance in \cite{Arenas;Rojas;2010, Biggers;Fried;1986, Carocca;Romero, Diaz;Donagi;1989, Guralnick;Shareshian;2007, Kanev;2006, Recillas;1994, Vetro;2007}.
Probably the most studied case is when $\psi:\mathcal{Y}\rightarrow \mathcal{X}$ is an unramified cover of degree two; the results obtained in this situation involve a systematic study of the Weyl groups,  see for instance \cite{Kanev;2006, Vetro;2007}.
\vspace{2mm}\\
In this paper we study and completely characterize  the Galois group of the cover $\varphi \circ \psi $ when  $ {{\psi}} $ is  a $q$-cyclic \'etale cover and ${ \varphi}$ is a cyclic $p$-gonal cover for any prime integers $q \ne p$  and $p$ odd. Using this characterization we give some results on isogenous decompositions of certain Prym and Jacobian varieties associated to intermediate subcovers given by appropriate  subgroups.
We mainly use the same notations and definitions as in \cite{bl,m}.

\section{An isogenous decomposition from the $q$-homology cover of a compact Riemann surface}
Associated to a compact Riemann surface ${\mathcal X}$ of genus $g$, and to each integer $q \geq 2$, is its $q$-homology cover Riemann surface $\widetilde{\mathcal X}$, whose Galois cover group is $H_{1}({\mathcal X};{\mathbb Z}_{q}) \cong {\mathbb Z}_{q}^{2g}$,
the $q$-homology group of $\mathcal X$ (see Section \ref{Sec:qhomology}).
In this section we provide an isogenous decomposition of the associated Prym variety $P(\widetilde{\mathcal X}/{\mathcal X})$ in terms of the  maximal subgroups of the $q$-homology group. In the next section, we will particularize this situation to the case when ${\mathcal X}$ is assumed to be cyclic $p$-gonal, for any odd prime integer $p \ne q.$

\subsection{The $q$-homology}\label{Sec:qhomology}
Let $q \geq 2$ be an integer and $\mathcal X$  a compact Riemann surface of genus $g \geq 2$. By the Klein-Koebe-Poincar\'e uniformization theorem, we may assume that ${\mathcal X}={\mathbb H}^{2}/\Gamma$, where ${\mathbb H}^{2}$ denotes the upper-half plane and $\Gamma \cong \pi_{1}({\mathcal X})$ is a co-compact Fuchsian group acting on it.
Let $\Gamma_{q}$ be the normal subgroup of $\Gamma$ generated by its derived subgroup $\Gamma'$ and the $q^{th}$-powers of the elements of $\Gamma$. Then $\widetilde{{\mathcal X}}={\mathbb H}^{2}/\Gamma_{q}$ is a compact Riemann surface of genus $\widetilde{g}=1+q^{2g}(g-1)$, called the $q$-\textit{homology cover} of ${\mathcal X}$. 
The abelian group $\widetilde{N}=\Gamma/\Gamma_{q} \cong H_{1}({\mathcal X};{\mathbb Z}_{q})\cong {\mathbb Z}_{q}^{2g}$, of conformal automorphisms of $\widetilde{{\mathcal X}}$, is  called the associated $q$-\textit{homology group}. In this case, ${\mathcal X}=\widetilde{{\mathcal X}}/\widetilde{N}$. Let $\pi_{\widetilde{N}}:\widetilde{{\mathcal X}} \to {\mathcal X}$ be the regular covering map with ${\rm deck}(\pi_{\widetilde{N}})=\widetilde{N}$. In \cite{Hidalgo:homology} it was proved that: (i) $\widetilde{N}$ is a normal subgroup of ${\rm Aut}(\widetilde{{\mathcal X}})$, and (ii) if ${\mathcal X}$ admits no conformal automorphisms of order a prime divisor of $q$, then $\widetilde{N}$ is the unique subgroup of ${\rm Aut}(\widetilde{\mathcal X})$ acting freely and isomorphic to ${\mathbb Z}_{q}^{2g}$. In particular, in this sitaution, if $q$ is a prime integer, then 
every $q$-subgroup of ${\rm Aut}(\widetilde{{\mathcal X}})$ is a subgroup of $\widetilde{N}$. 
As $\Gamma_{q}$ is a characteristic subgroup of $\Gamma$, for every $\Phi \in {\rm Aut}({\mathcal X})$ there is some $\widetilde{\Phi} \in {\rm Aut}(\widetilde{{\mathcal X}})$ such that $\pi_{\widetilde{N}} \circ \widetilde{\Phi}=\Phi \circ \pi_{\widetilde{N}}$. This gives a short exact sequence
\begin{equation} \label{eq:rho}
1 \to \widetilde{N} \to {\rm Aut}(\widetilde{{\mathcal X}}) \stackrel{\rho}{\to} {\rm Aut}({\mathcal X}) \to 1,
\end{equation}
where $\pi_{\widetilde{N}} \circ \widetilde{\Phi}=\rho(\widetilde{\Phi)} \circ  \pi_{\widetilde{N}}$ for each $\widetilde{\Phi} \in {\rm Aut}(\widetilde{{\mathcal X}})$.

The $q$-homology cover of ${\mathcal X}$ is its maximal cover with deck group isomorphic to ${\mathbb Z}_{q}^{n}$ for some positive $n$, as the next result shows.

\begin{thm}[\cite{Hidalgo:homology}]\label{teo:cubriente}
Let ${\mathcal S}$ be a compact Riemann surface, $F<{\rm Aut}({\mathcal S})$, $F \cong {\mathbb Z}_{q}^{n}$ acting freely on ${\mathcal S}$, and $P:{\mathcal S} \to {\mathcal X}$ a regular (unbranched) covering map with ${\rm deck}(P)=F$. Then there exist $\widetilde{L}<\widetilde{N}$ and a regular covering map $\pi_{\widetilde{L}}:\widetilde{{\mathcal X}} \to S$ with ${\rm deck}(\pi_{\widetilde{L}})=\widetilde{L}$, such that $\pi_{\widetilde{N}}=P \circ \pi_{\widetilde{L}}$; in particular, $F=\widetilde{N}/\widetilde{L}$.
\end{thm}

Particularizing Theorem \ref{teo:cubriente} to the cyclic situation ($n=1$) gives the following fact.

\begin{cor} \label{cor:cyclic}
Let ${ {\psi}} : {\mathcal Y} \to {\mathcal X}$ be a $q$-cyclic \'etale cover of ${\mathcal X}$. Then there exists a  maximal subgroup $\widetilde{L}$ of $\widetilde{N}$  such that ${\mathcal Y} = \widetilde{{\mathcal X}}/\widetilde{L}$.	
\end{cor}

\begin{rem}
Let us consider the homology cover ${\mathcal X}_{h} := \mathbb{H}/\Gamma'$. Then ${\mathcal X}_{h} \to {\mathcal X}$ is a Galois cover, with deck group (isomorphic to) $H_1({\mathcal X}, \mathbb{Z})$,  see for instance \cite{Maskit}. Let us start by noting that 
every $q$-cyclic \'etale  cover ${\mathcal Y} \to {\mathcal X}$ corresponds to a surjective homomorphism
$\theta_{\mathcal Y} : H_1({\mathcal X}, \mathbb{Z}) \rightarrow \mathbb{Z}_q$. In this setting, ${\mathcal X}_h$ is a Galois cover of ${\mathcal Y}$ with deck group $\ker(\theta_{\mathcal Y})$.
Observe that post-composing $\theta_{\mathcal Y}$ with an automorphism of $\mathbb{Z}_q$ keeps invariant $\ker(\theta_{\mathcal Y})$.
Fixing an (ordered) basis $\{ (a_1 , \ldots , a_{2g}) \}$ for $H_1({\mathcal X}, \mathbb{Z})$ and a generator $\gamma$ for $\mathbb{Z}_q$, by  applying $\theta_{\mathcal Y}$ to each $a_j$: $\theta_{\mathcal Y}(a_j) = \gamma^{l_j}$ and considering the $2g$-tuple $(l_1 , \ldots , l_{2g})$, we obtain a point in $\mathbb{Z}_q^{2g} \setminus \{0\}$ (note that the maximum common divisor of all these values $l_{j}$ and $q$ is equal to $1$). Post-composition with an automorphism of $\mathbb{Z}_q$ gives a multiple of $(l_1 , \ldots , l_{2g})$. This process thus assigns to the  $q$-\'etale cyclic cover ${\mathcal Y} \to {\mathcal X}$ the (cyclic of order $q$) subgroup  of  $\mathbb{Z}_q^{2g}$ generated by   $(l_1 , \ldots , l_{2g})$. Conversely, given a cyclic subgroup of order $q$ of  $\mathbb{Z}_q^{2g}$, say generated by  $(l_1 , \ldots , l_{2g})$ in  $\mathbb{Z}_q^{2g} \setminus \{0\}$, we define $\theta : H_1({\mathcal X}, \mathbb{Z}) \rightarrow \mathbb{Z}_q$ 
by $\theta(a_j) = \gamma^{l_j}$, and obtain the $q$-cyclic \'etale cover ${\mathcal Y}:={\mathcal X}_h/\ker(\theta) \to {\mathcal X}$.
\end{rem}

\subsection{An isogenous  decomposition of the Prym variety associated to the $q$-homology cover}
Let us now assume $q \geq 2$ to be a prime integer.
It is well known that each $q$-cyclic \'etale  cover ${ {\psi}} :{\mathcal Y} \to {\mathcal X}$ is naturally associated to a subgroup $\ker{{\psi}} ^* \cong {\mathbb Z}_{q}$ of the group $J{\mathcal X}[q]$ of points of order $q$ in the Jacobian variety $J{\mathcal X}$ of ${\mathcal X}$ (and conversely), where ${{\psi}} ^* : J{\mathcal X} \to J{\mathcal Y}$ is the pull-back map. 
Thus, the associated $q$-homology group $\widetilde{N}$ may be identified canonically with the group $J{\mathcal X}[q]$ of points of order $q$ in the Jacobian variety $J{\mathcal X}$ of ${\mathcal X}$. In this way, the $q$-homology cover $\widetilde{{\mathcal X}}$ maybe thought of as the curve where $J{\mathcal X}[q]$ acts with quotient ${\mathcal X}$.
This relationship is made clearer in the following result, where we also give the isogenous decomposition of $J\widetilde{{\mathcal X}}$ with respect to the action of $\widetilde{N}$.  Let us recall that the number of  maximal  subgroups of $\widetilde{N} \cong {\mathbb Z}_{q}^{2g}$ is $m_{q,g} = \dfrac{q^{2g}-1}{q-1}$.

\begin{thm} \label{producto1}
Let ${\mathcal X}$ be a compact Riemann surface of genus $g \geq 2$, and consider its $q$-homology cover $\widetilde{{\mathcal X}}$, with associated $q$-homology group $\widetilde{N} \cong \mathbb{Z}_q^{2g} \leq {\rm Aut}(\widetilde{{\mathcal X}}).$

Let $\mathcal{L} =  \{\widetilde{L}_1 , \ldots , \widetilde{L}_{m_{q,g}}\}$ be the set of  maximal  subgroups of $\widetilde{N}. $ For each $\widetilde{L}_{j} \in {\mathcal L}$ consider  ${\mathcal Y}_j := \widetilde{{\mathcal X}}/\widetilde{L}_j$ and $P({\mathcal Y}_{j}/{\mathcal X})$ the  Prym variety associated to the corresponding $q$-cyclic \'etale cover 
 $\psi_{j}:{\mathcal Y}_{j} \to {\mathcal X}$. \\
Then 	the isotypical decomposition of $J\widetilde{{\mathcal X}}$ with respect to the action of $\widetilde{N}$ is given by 
$$
J\widetilde{{\mathcal X}} \cong_{\textup{isog}} J{\mathcal X} \times     \prod_{j=1}^{m_{q,g}} P({\mathcal Y}_j/{\mathcal X})
  \ \ \text{  and }  \ \   P(\widetilde{{\mathcal X}}/{\mathcal X}) \cong_{\textup{isog}} \prod_{j=1}^{m_{q,g}} P({\mathcal Y}_j/{\mathcal X}).
$$

\end{thm}

\begin{proof}
Since $\pi_{\widetilde{N}}^*(J{\mathcal X}) \cong_{\textup{isog}} J{\mathcal X}$, then 
$J\widetilde{{\mathcal X}} \cong_{\textup{isog}} J{\mathcal X} \times    P(\widetilde{{\mathcal X}}/{\mathcal X})$.
Observe that this first decomposition is also $\widetilde{N}$-equivariant, since  $\pi_{\widetilde{N}}^*(J{\mathcal X})$ is the maximal abelian subvariety of $J\widetilde{{\mathcal X}}$ where the trivial representation $\chi_0$ of $\widetilde{N}$ acts, and therefore all non-trivial rational irreducible  representations $W$ of $\widetilde{N}$ act on the complementary subvariety $P(\widetilde{{\mathcal X}}/{\mathcal X})$ of $\pi_{\widetilde{N}}^*(J{\mathcal X})$ in $J\widetilde{{\mathcal X}}$. 
Furthermore, it is well known (see for instance \cite{CR}) that then there is an isogenous $\widetilde{N}$-equivariant decomposition  
 $$
P(\widetilde{{\mathcal X}}/{\mathcal X}) \cong_{\textup{isog}}  \prod_{W \in {\rm Irr}_{\mathbb{Q}}(\widetilde{N}), W\neq \chi_0} A_W ,
$$
where $\widetilde{N}$ acts on each canonically defined  abelian subvariety $A_W$ of  $P(\widetilde{{\mathcal X}}/{\mathcal X})$ by (an appropriate multiple of) the representation $W$. 
Since $\widetilde{N} \cong \mathbb{Z}_q^{2g}$ is an abelian group, a precise description of each $A_W$ is given as follows. 
According to \cite{CLR3} Corollary 4.2, each non trivial complex irreducible linear character $V$ of $\widetilde{N}$ is defined over the field $\mathbb{Q}[w_q]$, with $w_q$ a primitive $q$-th root of unity, and each non trivial rational irreducible  representation $W$ of $\widetilde{N}$ is given by 
$$
W = \bigoplus_{\sigma \in {\rm Gal}(\mathbb{Q}[w_q]/\mathbb{Q})} V^{\sigma} , 
$$
for some  non trivial linear character $V$ of $\widetilde{N}$ and its Galois conjugates $V^{\sigma}$. Furthermore, all characters $V^{\sigma}$ share the same kernel $\widetilde{L}$,  a subgroup of  $\widetilde{N}$ with cyclic quotient; thus, in our case, a  maximal  subgroup of $\widetilde{N}$, that is,  $\widetilde{L} \in {\mathcal L}$. Conversely, to each  maximal  subgroup $\widetilde{L} \in {\mathcal L}$  there corresponds a non trivial rational irreducible  representation $W$ of $\widetilde{N}$.
Furthermore, it follows from Theorem 5.1 in \cite{CLR3} that in our case
$$
A_W = \pi_L^*(P({\mathcal Y}_{\widetilde{L}}/{\mathcal X})) \cong_{\textup{isog}} P({\mathcal Y}_{\widetilde{L}}/{\mathcal X}),
$$
where $\pi_{\widetilde{L}} : \widetilde{{\mathcal X}} \to {\mathcal Y}_{\widetilde{L}} := \widetilde{{\mathcal X}}/\widetilde{L}$ is the corresponding \'etale cover, thus concluding the proof.
\end{proof}

\section{The case of $p$-cyclic gonal curves}
In this section $q \geq 2$ is still a prime integer, and we assume ${\mathcal X}$ to be a cyclic $p$-gonal curve of genus $g \geq 2$, with $p \geq 3$  a prime  integer such that $ p \ne q$. In this way, ${\mathcal X}$ admits a conformal automorphism $\tau$ of order $p$ such that the Riemann orbifold ${\mathcal X}/\langle \tau \rangle$ has genus zero. By the Riemann-Hurwitz formula, the number of fixed points of $\tau$ is some integer $r \geq 3$ such that $g=(p-1)(r-2)/2$. 

Gilman observed in \cite{Gilman} that  there exists a nice basis for $H_{1}({\mathcal X};{\mathbb Z})$, called an adapted basis for the action of  the automorphism $\tau$, as follows.

\begin{lem}[Adapted basis \cite{Gilman}]\label{lema1}
Let ${\mathcal X}$ be a closed Riemann surface of genus $g=(r-2)(p-1)/2$, where $r \geq 3$ and $p$ is a prime integer.
If $\tau$ is an order $p$ conformal automorphism of ${\mathcal X}$ with exactly $r$ fixed points, and $\tau_{*}$ is the induced action of $\tau$ on $H_{1}({\mathcal X};{\mathbb Z})$, then 
there exists a basis $\left\{ a_{j, i } \right\}_{\substack{j =1, \dots, r-2\\i = 1, \dots, p -1  }}$ of $H_{1}({\mathcal X};{\mathbb Z})$ (it might not be a canonical one) such that 
$$\tau_{*}(a_{j,  1})=a_{j, 2}, \tau_{*}(a_{j, 2})=a_{j, 3}, \ldots, \tau_{*}(a_{j, p-2})=a_{j, p-1}, \tau_{*}(a_{j, p-1})=a_{j, p}, \tau_{*}(a_{j, p})=a_{j, 1} $$
where  $a_{j, p}=(a_{j, 1}a_{j, 2}\cdots a_{j, p-1})^{-1}$,  for each  $j=1,\ldots, r-2$.
\\
\end{lem}

As before, $\widetilde{{\mathcal X}}$ is the  $q$-homology cover of ${\mathcal X}$ and $\widetilde{N}<{\rm Aut}(\widetilde{{\mathcal X}})$ is the associated $q$-homology group.
We now show that the automorphism of order $p$ of ${\mathcal X}$ may be lifted to $\widetilde{{\mathcal X}}$. 

\begin{lem}\label{lema2}
Let $ \tau $ be an automorphism of order $p$ of ${\mathcal X}$. Then there exists $\widetilde{\Phi}\in {\rm Aut}(\widetilde{{\mathcal X}})$ of order $p$ such that $\pi_{\widetilde{N}} \circ {\widetilde\Phi}=\tau \circ \pi_{\widetilde{N}}$.
\end{lem}
\begin{proof}
We know the existence of some $\widetilde{\tau} \in {\rm Aut}(\widetilde{{\mathcal X}})$ such that $\pi_{N} \circ \widetilde{\tau}=\tau \circ \pi_{N}$ (that is, $\rho(\widetilde{\tau})=\tau$, with $\rho$ as in  \eqref{eq:rho}). If $\widetilde{\tau}$ has order $p$, then we are done, with 
$\widetilde{\Phi} = \widetilde{\tau}.$  If it does not, then as $\rho(\widetilde{\tau}^{p})=\tau^{p}=1$, it follows that $\widetilde{\tau}^{p} \in \widetilde{N}$. Hence $(\widetilde{\tau}^{q})^{p}=(\widetilde{\tau}^{p})^{q}=1$ and $\rho(\widetilde{\tau}^{q})=\tau^{q}$, which has order $p$, since $p \neq q$. Then  $\widetilde{\Phi} = \widetilde{\tau}^{q}$ has order $p.$
\end{proof}

If we denote by $\widetilde{P} := \langle \widetilde{\Phi} \rangle \cong {\mathbb Z}_{p}$, then  we have that $\widetilde{G}:=\langle \widetilde{N}, \widetilde{\Phi} \rangle = \widetilde{N} \rtimes \widetilde{P}$,  since  $\widetilde{N} \unlhd {\rm Aut}(\widetilde{{\mathcal X}}). $ Also,  the conjugation action of $\widetilde{\Phi} $ on $\widetilde{N}$ is the one induced by the action of $\tau$ in $H_{1}({\mathcal X};{\mathbb Z})$. In particular, Lemma \ref{lema1} asserts the following fact.

\begin{lem} \label{leman1}
There exists a set of generators $\left\{ n_{j, i } \right\}_{\substack{j =1, \dots, r-2\\i = 1, \dots, p -1  }} $ of $\widetilde{N}$  such that
$$\widetilde{\Phi}(n_{j,1})=n_{j, 2}, \widetilde{\Phi}(n_{j, 2})=n_{j, 3}, \ldots,  \widetilde{\Phi}(n_{j, p-1})=n_{j, p}, \widetilde{\Phi}(n_{j, p})=n_{j, 1}$$	
where  $n_{j, p}=(n_{j, 1}n_{j, 2}\cdots n_{j, p-1})^{-1}$, 
for each 
$ j=1,\ldots, r-2$.
\end{lem}	

As an immediate consequence we obtain the following result.

\begin{cor}\label{lema3}
\begin{enumerate}
\item If $ \; C $ is a cyclic subgroup of  $ \widetilde{G}$ such that  $ \widetilde{\Phi} \in C, $ then $ C = \widetilde{P}.$
In particular, $\widetilde{G}$ is a Frobenius group with kernel  $\widetilde{N}$ and complement  $ \widetilde{P}$.

\item If $x \in {\rm Fix}(\tau)$, then in the fiber $\pi_{\widetilde{N}}^{-1}(x)$ there is exactly one fixed point by $\widetilde{\Phi}$.
\vspace{2mm}\\
\end{enumerate} \end{cor}
Let us recall the following well known fact.

\begin{lem} \label{leman}
For any  $k \in \{1,2,\ldots,(p-1)(r-2)-1\}$, the number of subgroups $K < \widetilde{N}$ with  $\vert K \vert  = {q}^{k}$ is given by 

$$
\displaystyle{\frac{\displaystyle\prod_{j=0}^{k-1}  \left( q^{(p-1)(r-2)-j}-1\right)}{ \displaystyle\prod_{j=0}^{k-1} \left(q^{k-j}-1 \right)}}.
$$

\end{lem}	

\begin{lem}\label{lema5}
If $K<\widetilde{N}$ is invariant under conjugation by $\widetilde{\Phi}$, then $K \cong {\mathbb Z}_{q}^{s}$ for some $s \in \{0,1,\ldots,(p-1)(r-2)\}$ such that $q^s \equiv 1 \mod p$.
In fact, the minimum positive $s$ is the multiplicative order of $q$ in ${\mathbb Z}_{p}^*$, and all other values of $s$ are its (additive) multiples contained in $\{0,1,\ldots, (p-1)(r-2)\}$. 
\end{lem}

\begin{proof}
This is a direct consequence of Corollary \ref{lema3}, since the orbit of each element of $\widetilde{N} \setminus \{1\}$ under conjugation by  $\widetilde{\Phi}$ has cardinality $p$, and  therefore 
$$
q^{s}=|K| = 1+p \, (\textup{number of orbits}). 
$$
\end{proof}

\begin{rem} \label{rem:not}
It follows from Lemma \ref{lema5} that, for $(p,q-1)=1$, the  maximal subgroups $\widetilde{L}_{j} \in {\mathcal L}$ are not $\widetilde{\Phi}$-invariant.
\end{rem}

\begin{lem}\label{lema7}
For any  $s \in \{0,1,\ldots,(p-1)(r-2)\}$ such that $q^s \equiv 1 \mod p$ there exists  ${\mathbb Z}_{q}^{s}  \cong K<\widetilde{N}$ invariant under conjugation by $\widetilde{\Phi}$.
\end{lem}

\begin{proof}
Is clear that we may assume   $1 \leq s \leq (p-1)(r-2)-1$. 
According to  Lemma \ref{lema1}, for all $ 0 \leq k \leq r-2$ there exist $\widetilde{\Phi}$-invariant subgroups of $\widetilde{N}$ of order ${q}^{(p-1)k}$ . Now, if there exists $s$ such that $q^s \equiv 1 \mod p$ and $s = (p-1)k+n$ with $1 \leq n \leq p-2$, then $q^n \equiv 1 \mod p$, and hence it is enough to  verify our assertion for $1 \leq s \leq p-2$.  
So fix $s$ such that $1 \leq s \leq p-2$ and  $q^s \equiv 1 \mod p$.  If  $\mathcal{F}$ is the collection  of all subgroups $K \cong {\mathbb Z}_{q}^{s}$ of $\widetilde{N}$, then (by Lemma \ref{leman}) its cardinality is 

$$
\# {\mathcal F}:=\prod_{j=0}^{s-1}  \frac{q^{(p-1)(r-2)-j}-1}{q^{s-j}-1}.
$$

The subgroup $\widetilde{P}  = \langle \widetilde{\Phi} \rangle $ acts on $\mathcal{F}$ by conjugation, and it follows from  Lemma \ref{lema3} that for each $K$ in $\mathcal{F}$ either $K$ is invariant under conjugation   or the orbit of $K$ under the action has size $p$. Denote by $\textup{Fix}(\widetilde{\Phi})$ the set of $K$ in  $\mathcal{F}$ that are invariant under conjugation by $\widetilde{\Phi}$; we have to show that $\textup{Fix}(\widetilde{\Phi})$ is non-empty.  As $q^{(p-1)(r-2)}$ and (by hypothesis) $q^{s}$ are both congruent to $1$ module $p$, it follows that 
$\frac{q^{(p-1)(r-2)-j}-1}{q^{s-j}-1}=a_{j}/b_{j}$, where $a_{j}$ and $b_{j}$ are relatively prime integers, each one also not divisible by $p$. Now, since 
$\# {\mathcal F} = \vert \textup{Fix}(\widetilde{\Phi}) \vert + np$,
for some nonnegative integer $n$, and the left hand side of this equality is not divisible by $p$, we are done.   
\end{proof}

Observe that if $s_0$ is the multiplicative order of $q$ in $\mathbb{Z}_p^*$, then  there is $K <\widetilde{N}$ with $K \cong {\mathbb Z}_{q}^{s_0}$ and $K$ normal in $\widetilde{G}$,
according to Lemma \ref{lema7}.  Furthermore, $K$ is minimal normal in $\widetilde{G}$, since the action of conjugation by $\widetilde{\Phi}$ is irreducible on $K$.

\begin{lem} \label{lema8}
For every  $\widetilde{L} \in {\mathcal L}$ , the cardinality of its core $\widetilde{L}_{\widetilde{G}}$  in $\widetilde{G}$ is at least
$q^{(p-1)(r-3)}$.	
\end{lem} 

\begin{proof} 
Up to index permutation, we may assume that $\{\widetilde{L_1} := \widetilde{L} , \widetilde{L_2} , \ldots , \widetilde{L_p} \}$ denote the conjugate subgroups of $\widetilde{L}$ in $\widetilde{G}$, and  consider the homomorphism 
$$
f : \widetilde{N} \longrightarrow \widetilde{N}/\widetilde{L_1} \times \widetilde{N}/\widetilde{L_2} \times \ldots \times \widetilde{N}/\widetilde{L_p}\cong \mathbb{Z}_q^p
$$
given by $f(n)= (n\widetilde{L_1} , n\widetilde{L_2} , \ldots , n\widetilde{L_p})$.	
Then $\displaystyle{ \ker(f) = \bigcap_{j=1}^p \widetilde{L_j} =\widetilde{L}_{\widetilde{G}}}$ and hence $|\widetilde{N}|/|\widetilde{L}_{\widetilde{G}}| \leq q^p$, or, equivalently, $|\widetilde{L}_G| \geq q^{(p-1)(r-2)-p}$.
Since $|\widetilde{L}_{\widetilde{G}}|$ is a divisor of $|\widetilde{N}| =  q^{(p-1)(r-2)}$, the result follows.
\end{proof}	

\begin{thm}\label{galoiscover}
Let $ p $ be an odd prime integer and $ \varphi :  {\mathcal X} \rightarrow  \mathbb{P}^{1}$ a cyclic $p$-gonal cover. Let 
${\psi} : {\mathcal Y} \rightarrow  {\mathcal X}$ be a $q$-cyclic \'etale cover, where $q \neq p$ is a prime  integer such that  $(p,q-1)=1$.
Then 
\begin{enumerate}
\item the composition  $\varphi \circ \psi:{\mathcal Y} \to {\mathcal X} \to \mathbb{P}^1 $ is non-Galois.

\item  If $Z \to \mathbb{P}^1$ denotes the Galois closure of the composition $ \varphi \circ \psi,$ then its Galois group is
$$
G = N \rtimes P \cong (\mathbb{Z}_q)^s \rtimes\mathbb{Z}_p , 
$$
for some $s \in \{ 1, 2, \ldots , p-1\}$ such that $q^s \equiv 1 \mod p$. 
\end{enumerate}
\end{thm}

\begin{proof}
It follows from Theorem \ref{teo:cubriente} that the $q$-homology cover $\widetilde{{\mathcal X}}$ of ${\mathcal X}$ is a cover of ${\mathcal Y}$, and that there exists a  maximal subgroup $\widetilde{L}\in {\mathcal L}$ such that ${\mathcal Y} = \widetilde{{\mathcal X}}/\widetilde{L}$, ${\mathcal X}= \widetilde{{\mathcal X}}/\widetilde{N}$, and such that the cover $\varphi: {\mathcal Y} \to {\mathcal X}$ is given by the induced action of $\widetilde{N}/\widetilde{L}$ on ${\mathcal Y}$.
As $(p,q-1)=1$, it follows from Lemma \ref{lema5} that $\widetilde{L}$ is not $\widetilde{\Phi}$-invariant. Therefore the composition $\varphi \circ \psi : {\mathcal Y} \to {\mathcal X} \to \mathbb{P}^{1}$ is non-Galois.
Now consider the normal subgroup $K = \widetilde{L}_{\widetilde{G}}$ of $\widetilde{G}$ and set $G := \widetilde{G}/K := N \rtimes P \cong \mathbb{Z}_q^s \rtimes \mathbb{Z}_p $.
According to Lemma \ref{lema8} we have that  $s \in \{1, \ldots , p-1\}$ and  $q^s \equiv 1 \mod p$. Also, $L:= \widetilde{L}/K$ is a  maximal  subgroup of $N$ with $L_G = \{ 1\}$.
If ${\mathcal Z} :=  \widetilde{{\mathcal X}}/K$, then $G$ acts on ${\mathcal Z}$. Furthermore,   ${\mathcal Y} = {\mathcal Z}/L$, $\; {\mathcal X}={\mathcal Z}/N$, and the cover $ \varphi : {\mathcal Y} \to {\mathcal X}$  is given by the induced action of $N/L$ on ${\mathcal Y}$.
That is, we have the following commutative diagram
$$
\xymatrix{
& \widetilde{{\mathcal X}} \ar[d]_{} & \\	
& {\mathcal Z}= \widetilde{{\mathcal X}}/K  \ar[ddd]^{pq^{s}:1} \ar[dl]_{q^{s-1}:1} \ar[dr]^{p:1}  &\\
{\mathcal Y}={\mathcal Z}/L  \ar[d]_{q:1}  & & {\mathcal Z}/P \ar[ddl]^{q^s:1}\\
{\mathcal X}={\mathcal Z}/N \ar[dr]_{p:1} & &  \\
& \mathbb{P}^1 = {\mathcal Z}/G
}
$$ 
\end{proof}

\begin{rem}
Observe that $s$ depends on ${\mathcal Y}$, or, equivalently, on $\widetilde{L}$, and not only on $p$ and $q$. We illustrate with the next example.	
\end{rem}

\begin{exam}
If $q=3$, $p=13$ and $r=3$, then ${\mathcal X}$ has genus $6$, ${\mathcal Y}$ has genus $16$, $\widetilde{{\mathcal X}}$ has genus $2,657,206$ and $\widetilde{G} = \widetilde{N} \rtimes \widetilde{P}$, where 
$\widetilde{N}= \langle a_1 , a_2 , \ldots , a_{12}\rangle \cong \mathbb{Z}_3^{12}$ and $\widetilde{P} \cong \langle \widetilde{\Phi} \rangle = \mathbb{Z}_{13}$. The conjugation action of $\widetilde{\Phi} $ on $\widetilde{N}$ is given by
$\widetilde{\Phi}  a_{j} \widetilde{\Phi} ^{-1}=a_{j+1}$, for $j=1,\ldots,12$ and $\widetilde{\Phi}  a_{13} \widetilde{\Phi} ^{-1}=a_{1}$, where $a_{1}\cdots a_{12} a_{13}=1$. A  generating vector for the given action of $\widetilde{G}$ on $\widetilde{{\mathcal X}}$ is $(\widetilde{\Phi}  , \widetilde{\Phi}  a_1 , a_1\widetilde{\Phi} ^{11})$.
In this case the multiplicative order of $q$  in $ \mathbb{Z}_{13}^{\ast} $is $3$, and there are   maximal  subgroups of $\widetilde{N}$ with core $K$ of sizes $1$, $q^3$, $q^6$ and $q^{9}$. Examples for these  maximal  subgroups $\widetilde{L}_{j}$ with core $K_{j}$ are the following (computations were carried out in GAP \cite{GAP}):
$$\widetilde{L}_{1}=\langle  a_{1}, a_{2}, a_{3}, a_{4}, a_{5}, a_{6}, a_{7}, a_{8}, a_{9}, a_{10}, a_{11} \rangle, \; K_{1}=\{I\}.$$
$$\widetilde{L}_{2}=\langle a_{1}, a_{2}, a_{3}, a_{4}, a_{5}, a_{6}, a_{7}, a_{8}, a_{9}a_{12}^2, a_{10}, a_{11} \rangle,$$ 
$$K_{2}=\langle  a_{1}^{2}a_{5}a_{6}^{2}a_{7}a_{8}^{2}a_{10}a_{11}, a_{2}^{2}a_{6}a_{7}^{2}a_{8}a_{9}^{2}a_{11}a_{12}, a_{3}^{2}a_{4}a_{5}^{2}a_{6}a_{8}^{2}a_{9}^{2}a_{12}  \rangle  \cong {\mathbb Z}_{3}^{3}.$$
$$\widetilde{L}_{3}=\langle a_{1}, a_{2}, a_{3}, a_{4}, a_{5}, a_{6}a_{11}, a_{7}, a_{8}a_{11},a_{9}a_{11},  a_{10}a_{11}^{2}, a_{12}  \rangle,$$
$$K_{3}=\langle a_{1}a_{8}a_{10}^{2}a_{11}^{2}a_{12}^{2}, a_{2}^{2}a_{3}a_{8}a_{9}^{2}a_{12}^{2}, a_{3}^{2}a_{4}^{2}a_{8}a_{9}^{2}a_{10}a_{11}^{2}a_{12}, a_{4}a_{5}a_{6}a_{7}a_{8}^{2}a_{10}^{2}a_{11}a_{12}^{2}, $$
 $$a_{5}^{2}a_{6}^{2}a_{7}a_{8}^{2}a_{9}^{2}a_{10}a_{11}^{2}a_{12}^{2}, a_{6}^{2}a_{8}a_{9}a_{10}a_{12}^{2}  \rangle \cong {\mathbb Z}_{3}^{6}.$$
$$\widetilde{L}_{4}=\langle  a_{1}, a_{2}, a_{3}a_{12}, a_{4}, a_{5}a_{12}, a_{6}a_{12},a_{7}a_{12}, a_{8}a_{12}^{2}, a_{9}a_{12}^{2}, a_{10}, a_{11}a_{12} \rangle,$$
$$K_{4}=\langle  
a_{1}^{2}a_{2}a_{3}a_{4}a_{6}a_{7}a_{8}a_{9}a_{10}a_{11}a_{12}^{2}, 
 a_{2}^{2}a_{3}a_{4}^{2}a_{5}^{2}a_{6}^{2}a_{7}^{2}a_{8}^{2}a_{9}^{2}a_{10}^{2}a_{11}^{2}a_{12}^{2}, 
 a_{3}a_{4}^{2}a_{12}, $$
 $$
 a_{4}a_{5}a_{7}^{2}, 
 a_{5}a_{6}^{2}a_{7}^{2}a_{12}^{2}, 
 a_{6}^{2}a_{7}a_{8}^{2}a_{12}, 
 a_{7}^{2}a_{8}^{2}a_{9}^{2}a_{12}, 
 a_{8}a_{9}^{2}a_{10}a_{11}^{2}a_{12}^{2}, 
 a_{9}^{2}a_{10}^{2}a_{12} 
  \rangle \cong {\mathbb Z}_{3}^{9}.$$

The corresponding Galois covers for the composite covers ${\mathcal Y}_{j} \to {\mathcal X} \to \mathbb{P}^1$ have respective deck groups $G_{1} \cong \mathbb{Z}_3^{12} \rtimes \mathbb{Z}_{13}$, $G_{2}\cong \mathbb{Z}_3^{9} \rtimes \mathbb{Z}_{13}$, $G_{3} \cong \mathbb{Z}_3^{6} \rtimes \mathbb{Z}_{13}$ and $G_{4} \cong \mathbb{Z}_3^{3} \rtimes \mathbb{Z}_{13}$, acting on ${\mathcal Z}_{j}=\widetilde{{\mathcal X}}/K_{j}$ of corresponding genus $2,657,206$; $98,416$; $3,646$ and $136$. 
\end{exam}

\section{The representations of $\widetilde{G}$}
As we know from Lemma \ref{lema3}, $\widetilde{G} = \widetilde{N} \rtimes \widetilde{P}$ is a Frobenius group with abelian kernel. Then, according to  \cite[Proposition 25]{s}, its complex irreducible representations are of two kinds: the first kind are those complex irreducible representations of $\widetilde{P}$, extended to $\widetilde{G}$ by composing with the quotient map $\widetilde{G} \to \widetilde{P}$; there are $p$ of these, all of degree one: we will denote them by $\chi_0$ for the trivial one, and $\chi_1 , \ldots , \chi_{p-1}$ for the non-trivial; for $j$ in $\{ 1, \ldots , p-1\}$, each $\chi_j$ has  kernel $\widetilde{N}$. 

The other kind are the representations of $G$ induced from the non-trivial representations of $\widetilde{N}$; recall from the proof of Theorem \ref{producto1} that there are $q^{(p-1)(r-2)}-1$ non-trivial complex irreducible representations of $\widetilde{N}$, all of degree one, and they are grouped into orbits $\{ V^{\sigma} : \sigma \in {\rm Gal}(\mathbb{Q}[w]/\mathbb{Q})\}$ of $q-1$ elements to give the rational irreducible representations of $\widetilde{N}$. There are natural bijections between the set of orbits under the Galois group action, the set ${\mathcal L}$ of  maximal  subgroups on $\widetilde{N}$, and the set of rational irreducible representations of $\widetilde{N}$. Recall that $m_{q,g} = \dfrac{q^{2g}-1}{q-1}$ is the cardinality of each of these sets.  

Each non-trivial complex irreducible  representation $V$ of $\widetilde{N}$ induces a complex  irreducible representation $\widetilde{V}$ of $\widetilde{G}$, of degree $[\widetilde{G}:\widetilde{N}]=p$. It is a general fact that if $\ker V = \widetilde{L}$, then $\ker \widetilde{V} = \widetilde{L}_{\widetilde{G}}$.

The subgroup $\widetilde{P}$ also acts on the set of non trivial complex irreducible representations of $\widetilde{N}$, by sending $V$ to $\mbox{}^{ \widetilde{\Phi}^j} V$, defined by its character $\chi_{\mbox{}^{ \widetilde{\Phi}^j} V}(n) = \chi_V( \widetilde{\Phi}^{-j} n  \widetilde{\Phi}^j)$ for each $n$ in $\widetilde{N}$ and $0 \leq j \leq p-1$.

It is clear that $\ker ( \mbox{}^{ \widetilde{\Phi}^j} V )= \widetilde{\Phi}^{j} (\ker V) \widetilde{\Phi}^{-j}$, $ \; \widetilde{  \mbox{}^{\widetilde{\Phi}^j} V } = \widetilde{V}$, and $( \mbox{}^{\widetilde{\Phi}^j} V)^{\sigma} =  \mbox{}^{\widetilde{\Phi}^j} (V^{\sigma})$.  

Equivalently, consider the action of conjugation by powers of $\widetilde{\Phi}$ on $\mathcal{L}$. As $(p,q-1)=1$, all its orbits are of length $p$. Up to permutation of the indices, we may assume that the following is a set of representatives of these orbits 
$${\mathcal F}:=\left\{ \widetilde{L}_{j} : 1 \leq j \leq t:=\dfrac{m_{q,g}}{p} = \dfrac{q^{(p-1)(r-2)}-1}{p(q-1)}\right\} \subset {\mathcal L}.$$ 

For each $\widetilde{L}_{j} \in {\mathcal F}$, choose a complex irreducible representation $V_j$ of $\widetilde{N}$ such that $\widetilde{L}_{j} = \ker(V_j)$.

With the above notation we have the following result.

\begin{prop}
The complex irreducible representations of $\widetilde{G}$ are given by $\chi_0, \ldots, \chi_{p-1}$ of degree one, and $\widetilde{V_j}$, $1\leq j \leq \dfrac{q^{(p-1)(r-2)}-1}{p}$, of degree $p$.
The rational irreducible representations  are given by $\chi_0$, $U = \chi_1 \oplus \ldots \oplus \chi_{p-1}$ of degree $p-1$, and $\displaystyle U_j = \bigoplus_{\sigma \in {\rm Gal}(\mathbb{Q}[w_q]/\mathbb{Q})}   (\widetilde{V_j})^{\sigma}$ of degree $p(q-1)$, 	for $1 \leq j \leq t$.
\end{prop}

\begin{proof}
Given a	non-trivial complex irreducible representation $V$ of $\widetilde{N}$, each conjugate of $V$ by a power of $\widetilde{\Phi}$ induces the same complex irreducible representation $\widetilde{V}$ of $\widetilde{G}$. So there are $\dfrac{q^{(p-1)(r-2)}-1}{p}$ complex irreducible representations of the second kind, of degree $p$.

To see that these $\widetilde{V}$ together with those of the first kind are all the complex irreducible representations of $\widetilde{G}$, add the squares of their degrees: 
$$
\underbrace{(1+1+ \ldots + 1)}_{p} + p^2\displaystyle{\left( \dfrac{q^{(p-1)(r-2)}-1}{p}\right)} = p q^{(p-1)(r-2)} = |\widetilde{G}|.
$$

It is clear that $\chi_0$ and  $U = \chi_1 \oplus \ldots \oplus \chi_{p-1}$ are rational irreducible representations of $\widetilde{G}$.

Since the field of definition of a non-trivial complex irreducible  representation $V$ of $\widetilde{N}$ is $\mathbb{Q}[w_q]$, with $w_q$ a primitive $q$-th root of unity, the same holds for its induced representation  $\widetilde{V}$, and the rational representation $U$ of $\widetilde{G}$ Galois associated to a complex irreducible representation of the form $\widetilde{V}$ is obtained by adding the corresponding $(q-1)$ Galois conjugates of $\widetilde{V}$.
\end{proof}

\begin{cor}
Consider the action of conjugation by powers of $\widetilde{\Phi} $ on the set $\mathcal{L}$ of  maximal subgroups of $\widetilde{N}$. 
Then there is a natural bijective correspondence among the set ${\mathcal F}$ of orbits in $\mathcal{L}$ and the set of rational irreducible representations of $\widetilde{G}$ with kernel properly contained in $\widetilde{N}$, and both sets have cardinality $t=\dfrac{q^{(p-1)(r-2)}-1}{p(q-1)}$. 
\end{cor}

 The following can be seen as a consequence of Theorem \ref{producto1}.

\begin{prop}
The isotypical decomposition of $J\widetilde{{\mathcal X}}$ under the action of $\widetilde{G}$ is given by
$$
J\widetilde{{\mathcal X}} \cong_{\textup{isog}} J{\mathcal X} \times \prod_{j=1}^{t} (P({\mathcal Y}_{j}/{\mathcal X}))^p , \; ({\mathcal Y}_{j}=\widetilde{{\mathcal X}}/\widetilde{L}_{j}, \; \widetilde{L}_{j} \in {\mathcal F}),
$$
where $\widetilde{G}$ acts on $J{\mathcal X}$ via $U$ and on $(P({\mathcal Y}_{j}/{\mathcal X}))^p$ via $U_j$.
\end{prop}

\begin{proof} If $ \; \widetilde{H} \; $ is any subgroup of $\widetilde{G}$ we will denote by $ \; \rho_{\widetilde{H}} \; $ the representation of $\widetilde{G}$ induced by the trivial representation of $ \; \widetilde{H}.$
It is easy to check that   $ \; \rho_{\widetilde{N}} = \chi_0 \oplus U \; $ and $ \; \rho_{\widetilde{L}_{j}} = \chi_0 \oplus U \oplus U_j, $ for all $ \; j \in \{1,\dots , t\}. \; $ Then the result follows applying  Corollary 5.4 of \cite{CR}.
\end{proof}

\section{On certain intermediate covers}
We now give a relation between certain subcovers of $\widetilde{{\mathcal X}}$ given as quotients by appropriate subgroups of $\widetilde{G}$.

First consider the regular covers $\pi_{\widetilde{P}} : \widetilde{{\mathcal X}} \to \widetilde{{\mathcal X}}/\widetilde{P}$ and $\pi_{\widetilde{N}} : \widetilde{{\mathcal X}} \to \widetilde{{\mathcal X}}/\widetilde{N} = {\mathcal X}$.  
According to Corollary \ref{lema3}, if $x \in {\rm Fix}(\tau) \subset {\mathcal X}$, then in the fiber $\pi_{\widetilde{N}}^{-1}(x)$ there is exactly one fixed point of $\widetilde{\Phi} $. Hence ${\mathcal T} := \widetilde{{\mathcal X}}/\widetilde{P}$  is a Riemann orbifold with exactly $r$ cone points (each one of order $p$) and, by the Riemann-Hurwitz formula, of genus
$$g_{{\mathcal T}}=\frac{((p-1)(r-2)-2)\left(q^{(p-1)(r-2)}-1\right)}{2p}.$$

We know that the set of  maximal  subgroups of $\widetilde{N}$ modulo the conjugation action by $\phi$ has cardinality $t = \dfrac{q^{(p-1)(r-2)}-1}{p(q-1)}$.
Choose a set of representatives  $\mathcal{F}:=\{ \widetilde{L}_{1},\ldots, \widetilde{L}_{t}\}$ of these $t$ conjugacy classes.  For each $\widetilde{L}_{j}$,  consider the closed Riemann surface ${\mathcal Y}_{j}=\widetilde{{\mathcal X}}/\widetilde{L}_{j}$. On ${\mathcal Y}_{j}$ we have the \'etale action of the induced group $L_{j} := \widetilde{N}/\widetilde{L}_{j} \cong {\mathbb Z}_{q}$ such that ${\mathcal X}={\mathcal Y}_{j}/L_{j}$. By the Riemann-Hurwitz formula, the genus of ${\mathcal Y}_{j}$ is
$$g_{{\mathcal Y}_{j}}=1+\frac{q((p-1)(r-2)-2)}{2}$$
and the Prym variety ${\rm P}({\mathcal Y}_{j}/{\mathcal X})$ associated the the covering ${\mathcal Y}_{j} \to {\mathcal X}$ has dimension
$${\rm dim}\; {\rm P}({\mathcal Y}_{j}/{\mathcal X})=(g-1)(q-1)=\frac{((p-1)(r-2)-2)(q-1)}{2}.$$

The next result follows immediately, (this generalizes  Corollary 2.9 of  \cite{CLR} for the case $q=2$).
\begin{cor}\label{coro1}
$\sum_{j=1}^{t} {\rm dim}\; {\rm P}({\mathcal Y}_{j}/{\mathcal X})={\rm dim} J{\mathcal T}.$
\end{cor}

The next main result of this paper is the following, (see Theorem 3.1  of  \cite{CLR}).

\begin{thm}\label{teo1}
There is an isogeny 
$$\alpha:\prod_{j=1}^{t} {\rm P}({\mathcal Y}_{j}/{\mathcal X}) \to J{\mathcal T},$$
whose kernel is contained in the $q^{(p-1)(r-2)-1}$-torsion points.
\end{thm}

\subsection{Example}
Let us consider $r=3$, $p \geq 5$ and $q \geq 2$ prime integers so that $p \neq q$ and $(p,q-1)=1$. In this case, ${\mathcal X}$ is a Belyi curve of genus $g=(p-1)/2$, defined by an algebraic curve of the form $y^{p}=x(x-1)^{\alpha}$, some $\alpha \in \{1,\ldots,p-2\}$, and $\tau(x,y)=(x,\omega_{p}y)$, where $\omega_{p}=e^{2 \pi i/p}$. The $q$-homology cover $\widetilde{{\mathcal X}}$ has genus $\widetilde{g}=1+q^{p-1}(p-3)/2$ and the $q$-homology group is $\widetilde{N} \cong {\mathbb Z}_{q}^{p-1}$. The number of $\widetilde{\Phi}$-conjugacy classes of  subgroups $L_{j} \cong {\mathbb Z}_{q}^{p-2}$ of $\widetilde{N}$ is $t=(q^{p-1}-1)/p(q-1)$. Let $\widetilde{L}_{1},\ldots,\widetilde{L}_{t}$ be representatives of these conjugacy classes, and set ${\mathcal Y}_{j}=\widetilde{{\mathcal X}}/\widetilde{L}_{j}$. 

Then ${\mathcal Y}_{j}$ has genus $g_{{\mathcal Y}_{j}}=1+q(p-3)/2$ and the dimension of the Prym variety $P({\mathcal Y}_{j}/{\mathcal X})$ is $(p-3)(q-1)/2$. The Riemann orbifold ${\mathcal T}=\widetilde{{\mathcal X}}/\widetilde{P}$ has genus $g_{{\mathcal T}}=(p-3)(q^{p-1}-1)/2p$. 
Theorem \ref{teo1} asserts that we have an isogeny
$$J{\mathcal T} \cong_{\textup{isog}} \prod_{j=1}^{t} P({\mathcal Y}_{j}/{\mathcal X}).$$

If we consider the lowest value of $p=5$, then $g=2$, $\widetilde{g}=1+q^4$, $g_{{\mathcal T}}=(q^{4}-1)/5$, $t=(q^{4}-1)/5(q-1)$, $g_{{\mathcal Y}_{j}}=1+q$, 
${\rm dim}\; P({\mathcal Y}_{j}/{\mathcal X})=q-1$.
If moreover, we consider $q=2$, then $t=3$, each $P({\mathcal Y}_{j}/{\mathcal X})$ is of genus one and ${\mathcal T}$ of genus $3$. So, the Jacobian variety $J{\mathcal T}$, of dimension $3$, is isogenous to the product of three elliptic curves.

\section{Proof of Theorem \ref{teo1}}
\subsection{The isogeny $\alpha$}
For  each $j \in\{1, \dots, t\}$ we set
$$
\nu_j: \widetilde{{\mathcal X}} \rightarrow {\mathcal Y}_j \quad \mbox{and} \quad \mu: \widetilde{{\mathcal X}} \rightarrow {\mathcal T},
$$
the covering maps with ${\rm deck}(\nu_{j})=\widetilde{L}_{j}$ and ${\rm deck}(\mu)= \widetilde{P}$.
Then we have the corresponding induced homomorphisms
$$\nu_i^*:J{\mathcal Y}_i \rightarrow J\widetilde{{\mathcal X}} \quad \mbox{and} \quad {\rm Nm} \mu: J\widetilde{{\mathcal X}} \rightarrow J{\mathcal T}$$ 
between the corresponding Jacobians.

Now the addition map gives a homomorphism
\begin{equation} \label{e2.1}
\alpha:= \sum_{i=1}^{t} {\rm Nm} \mu \circ \nu_i^*{_{\vert P({\mathcal Y}_i/{\mathcal X})}} : \; \prod_{i=1}^t  P({\mathcal Y}_i/{\mathcal X}) \rightarrow J{\mathcal T}.
\end{equation}

According to Corollary \ref{coro1}, $\displaystyle\prod_{i=1}^t P({\mathcal Y}_i/{\mathcal X})$ and $J{\mathcal T}$ are of the same dimension.
In order to prove that $\alpha : \displaystyle\prod_{i=1}^t  P({\mathcal Y}_i/{\mathcal X}) \rightarrow J{\mathcal T}$ is an isogeny, with kernel contained in the $q^{(p-1)(r-2)-1}$-torsion points, we use the following result (for the proof see \cite[Corollary 2.7]{RR}).

\begin{prop} \label{p2.3}
Let $f: {\mathcal Z} \rightarrow {\mathcal X} := {\mathcal Z}/N$ be a Galois cover of smooth projective curves with Galois group $N$  and $H \subset N$ be a subgroup. 
Denote by $\nu: {\mathcal Z} \rightarrow {\mathcal Y} := {\mathcal Z}/H$  the corresponding cover. 
If $\{g_1, \dots, g_r\}$ is a complete set of representatives of $N/H$, then 
$$\nu^*(P({\mathcal Y}/{\mathcal X})) = \{ z \in J{\mathcal Z}^H \;|\; \sum_{i=1}^r g_i(z) = 0 \}^0.$$
\end{prop}

Now, for each $j \in \{1, \dots, t\}$ let us set 
$$
A_j := \nu_j^*(P({\mathcal Y}_j/{\mathcal X})), \;
A := \sum_{i=1}^t A_i \quad \mbox{and} \quad B:= \mu^*(J{\mathcal T}).
$$

Recall that $\widetilde{G} = \widetilde{N} \rtimes \widetilde{P}$, with
$$
\widetilde{N} = \left\{ \prod_{i=1}^{(p-1)(r-2)} s_i^{n_i} \;|\; 0 \leq n_i \leq q-1, i =1, \dots , (p-1)(r-2) \right\}
\; \mbox{and} \; \widetilde{P} = \langle \phi \rangle,
$$
where, by Lemma \ref{lema1} (Gilman's adapted basis), there is 
a disjoint decomposition into $(r-2)$ subcollections 
$\{s_{j_{1}},\ldots, s_{j_{p-1}}\}$, $j=1,\ldots, r-2$, such that, for each $j$, it holds that, if we set $s_{j_{p}}=(s_{j_{1}}s_{j_{2}}\cdots s_{j_{p-1}})^{-1}$, then the  conjugation action $\sigma^{*}$ induced by $\sigma$ on $N$ is given by 
\begin{equation}\label{accion}
\sigma_{*}(s_{j_{1}})=s_{j_{2}}, \sigma_{*}(s_{j_{2}})=s_{j_{3}}, \ldots, \sigma_{*}(s_{j_{p-2}})=s_{j_{p-1}}, \sigma_{*}(s_{j_{p-1}})=s_{j_{p}}, \sigma_{*}(s_{j_{p}})=s_{j_{1}}.
\end{equation}

The following result is clear.

\begin{lem} \label{lemm:lemaL}
Let $\widetilde{L}_1$ and $\widetilde{L}_2$ be two different  maximal  subgroups of $\widetilde{N}$. Then $[\widetilde{L}_1 : \widetilde{L}_1 \cap \widetilde{L}_2] = q$, and if $\{ n_1=1, n_2, n_2^2 , \ldots , n_2^{q-1}\}$ is a complete set of representatives of $\widetilde{N}/\widetilde{L}_1$, then the same set is a complete set of representatives of $\widetilde{L}_2/\widetilde{L}_1 \cap \widetilde{L}_2$.  
\end{lem}

Recall furthermore that $\widetilde{L}_j$ is the subgroup of $\widetilde{N}$ giving the cover $\widetilde{{\mathcal X}}  \rightarrow {\mathcal Y}_{j}$, and thus determines the covering ${\mathcal Y}_{j} \rightarrow {\mathcal X}$. Then it is easy to see that we have the following commutative diagram
{\small
\begin{equation} \label{d2.5}
\xymatrix{ 
&A \ar[rr]^{\displaystyle \sum_{i=0}^{p-1}  \widetilde{\Phi}^i}  \ar[dr]_{{\rm Nm} \mu} && B \ar[rr]^{\displaystyle \sum_{i=1}^t \sum_{h \in \widetilde{L}_i} h} \ar[dr]_{\beta} && A\\
\displaystyle\prod_{i=1}^t P({\mathcal Y}_i/{\mathcal X})   \ar[rr]_{\alpha} \ar[ur]^{ \sum_{i=1}^t \nu_i^*} && J{\mathcal T} \ar[rr]_{\beta \circ \mu^*} \ar[ur]_{\mu^*} && \displaystyle \prod_{i=1}^t P({\mathcal Y}_i/{\mathcal X}) \ar[ur]_{{\sum_{i=1}^t \nu_i^*}}
}
\end{equation}
}
with $\beta = ({\rm Nm} \nu_1, {\rm Nm} \nu_2, \dots, {\rm Nm} \nu_t)$.

For $i = 1, \dots,t$, we consider the following sub-diagram
\begin{equation} \label{d2.4}
\xymatrix{
&A_i \ar[rr]^{\sum_{i=0}^{p-1}   \widetilde{\Phi}^i}  \ar[dr]_{{\rm Nm} \mu} && B_i \ar[rr]^{\sum_{h \in \widetilde{L}_i} h} \ar[dr]_{{\rm Nm} \nu_i} && A_i\\
P({\mathcal Y}_i/{\mathcal X})   \ar[rr]_{\alpha_i} \ar[ur]^{ \nu_i^*} && C_i \ar[rr]_{{\rm Nm} \nu_i \circ \mu^*} \ar[ur]_{\mu^*} && P({\mathcal Y}_i/{\mathcal X}) \ar[ur]_{\nu_i^*}
}
\end{equation}
with $\alpha_i:= {\rm Nm} \mu \circ \nu_i^*, \; C_i := {\rm Nm} \mu(A_i)$ and $B_i := \mu^*(C_i)$.

\begin{prop}  \label{prop3.3}
For $i = 1, \dots, t$ the map ${\rm Nm} \nu_i \circ \mu^* \circ \alpha_i : P({\mathcal Y}_i/{\mathcal X}) \rightarrow P({\mathcal Y}_i/{\mathcal X})$ is multiplication by $q^{(p-1)(r-2)-1}$.
\end{prop}

\begin{proof}
Since $\nu_i^*: P({\mathcal Y}_i/{\mathcal X}) \rightarrow A_i$ is an isogeny, it suffices to show that the composition
$$\Phi_i: = \sum_{h \in \widetilde{L}_i} h \circ \sum_{i = 0} ^{p-1}  \widetilde{\Phi}^i: A_i \rightarrow A_i$$
is multiplication by $q^{(p-1)(r-2)-1}$. Now from Proposition \ref{p2.3} we deduce that
\begin{equation} \label{e2.7}
A_i = \{ z \in J\widetilde{{\mathcal X}} \;|\; hz = z \; \mbox{for all} \; h \in \widetilde{L}_i \; \mbox{and} \; 0 = \sum_{j=0}^{q-1} n_i^j \, z \; \}^0,
\end{equation}
where $\{ 1, n_i, n_i^2 , \ldots , n_i^{q-1}\}$  is a complete set of representatives of  $\widetilde{N}/\widetilde{L}_i$.

Now for any $z \in A_i$,
$$\Phi_i(z) = \sum_{h \in \widetilde{L}_i} h(z) + \sum_{h \in \widetilde{L}_i} h \sum_{k=1}^{p-1}  \widetilde{\Phi}^k(z).$$
By \eqref{e2.7}  we have
$$\sum_{h \in \widetilde{L}_i} h (z) = |\widetilde{L}_i| z = q^{(p-1)(r-2)-1}z,$$
and, for $k = 1,\dots,p-1$,
$$\sum_{h \in \widetilde{L}_i} h  \widetilde{\Phi}^k(z) =  \widetilde{\Phi}^k \sum_{h \in  \widetilde{\Phi}^{-k} \widetilde{L}_i  \widetilde{\Phi}^k} h(z) = 0,$$
since $\widetilde{L}_i \neq  \widetilde{\Phi}^{-k} \widetilde{L}_i  \widetilde{\Phi}^k$ and it follows from Lemma \ref{lemm:lemaL} that
$$\sum_{h \in  \widetilde{\Phi}^{-k} \widetilde{L}_i  \widetilde{\Phi}^k} h(z) = \sum_{j=0}^{q-1} n_i^j \left( \sum_{f \in \widetilde{L}_i \cap  \widetilde{\Phi}^{-k} \widetilde{L}_i  \widetilde{\Phi}^k} f(z)\right) = |\widetilde{L}_i \cap  \widetilde{\Phi}^{-k} \widetilde{L}_i  \widetilde{\Phi}^k| \sum_{j=0}^{q-1} n_i^j z =0.$$
	
Together this completes the proof of the proposition.
\end{proof}

Now, we proceed to finish the proof of the theorem.
Since
$$
\beta \circ \mu^* \circ \alpha = \prod_{i=1}^t ({\rm Nm} {\nu_i} \circ \mu^* \circ \alpha_i),
$$
Proposition \ref{p2.3} implies that $\beta \circ \mu^* \circ \alpha$ is multiplication by $q^{(p-1)(r-2)}$.
In particular $\alpha$ has finite kernel. But according to Corollary \ref{coro1},
$\displaystyle\prod_{i=1}^t P({\mathcal Y}_i/{\mathcal X})$ and $J{\mathcal T}$ have the same dimension. So $\alpha$ is an isogeny.


\end{document}